\documentclass[a4paper, 11pt]{amsart}
\usepackage[
  margin=1.9cm,
  includefoot,
  footskip=30pt,
]{geometry}
\usepackage{amsmath,amssymb,amscd}
\usepackage{amsfonts}
\usepackage[english]{babel}
\usepackage[leqno]{amsmath}
\usepackage{amssymb,amsthm}
\usepackage{mathrsfs} 
\usepackage{amscd}
\usepackage{enumerate}
\usepackage{epsfig}
\usepackage{relsize}
\usepackage{layout}
\usepackage[usenames,dvipsnames]{xcolor}
\usepackage[backref=page]{hyperref}
\usepackage{tikz}
\hypersetup{
 colorlinks,
 citecolor=Green,
 linkcolor=Red,
 urlcolor=Blue}
\usepackage[matrix,arrow,tips,curve]{xy}
\input{xy}
\xyoption{all}
\definecolor{darkgreen}{rgb}{0.0, 0.7, 0.0}
\definecolor{purple}{rgb}{0.5, 0.0, 0.5}
\definecolor{red}{rgb}{0.8, 0.2, 0.0}

\newtheorem{thm}{Theorem}[section]
\newtheorem{bthm}{Theorem}

\newtheorem{lemma}[thm]{Lemma}

\newtheorem{prop}[thm]{Proposition}

\numberwithin{equation}{section}
\theoremstyle{definition}
\newtheorem{defi}[thm]{Definition}

\theoremstyle{remark}
\newtheorem{remark}[thm]{Remark}

\newtheorem{example}[thm]{Example}
\newcommand{\Z}{\mathbb{Z}}

\newcommand{\Pic}{\operatorname{Pic}}

\def \Im{{\rm Im}}

\def \P{\mathbb{P}}

\def \F{\mathcal F}

\def\I{\mathcal I}
\def \L{\mathcal L}
\def \E{\mathcal E}
\def \G{\mathcal G}

\def\O{\mathcal O}

\def\M0{\mathcal M^0}

\DeclareMathOperator{\Proj}{{Proj}}
\DeclareMathOperator{\Sym}{{Sym}}

\title[Non-big Ulrich bundles: on quadrics and of small numerical dimension]{Non-big Ulrich bundles: the classification on quadrics and the case of small numerical dimension}

\author[A.F. Lopez, R. Mu\~{n}oz and J.C. Sierra]{Angelo Felice Lopez*, Roberto Mu\~{n}oz and Jos\'e Carlos Sierra}

\address{\hskip -.43cm Angelo Felice Lopez, Dipartimento di Matematica e Fisica, Universit\`a di Roma
Tre, Largo San Leonardo Murialdo 1, 00146, Roma, Italy. e-mail {\tt lopez@mat.uniroma3.it}}

\address{\hskip -.43cm Roberto Mu\~{n}oz, Departamento de Matem\'atica Aplicada a las TIC, ETSISI Universidad Polit\'ecnica de Madrid. C/ Alan Turing s/n. 28031, Madrid, Spain. email: {\tt roberto.munoz@upm.es}}

\address{\hskip -.43cm Jos\'e Carlos Sierra, Departamento de Matem\'aticas Fundamentales, Facultad de Ciencias, UNED,  C/ Juan del Rosal 10, 28040 Madrid, Spain. e-mail {\tt jcsierra@mat.uned.es}}

\thanks{* Research partially supported by  PRIN ``Advances in Moduli Theory and Birational Classification'' and GNSAGA-INdAM}

\thanks{{\it Mathematics Subject Classification} : Primary 14J60. Secondary 14J40, 14N05.}

\begin{document} 

\begin{abstract} 
On any smooth $n$-dimensional variety we give a pretty precise picture of rank $r$ Ulrich vector bundles with numerical dimension at most $\frac{n}{2}+r-1$. Also, we classify non-big Ulrich vector bundles on quadrics and on the Del Pezzo fourfold of degree $6$.  
\end{abstract}

\maketitle

\section{Introduction}

Let $X \subseteq \P^N$ be a smooth irreducible complex closed variety of dimension $n \ge 1$. In this paper we consider the investigation of positivity properties of Ulrich vector bundles on $X$. Recall that a vector bundle $\E$ on $X$ is Ulrich if $H^i(\E(-p))=0$ for all $i \ge 0$ and $1 \le p \le n$. The importance of Ulrich vector bundles and the consequences on the geometry of $X$ are well described for example in \cite{es, b1, cmp} and references therein. 

It was highlighted in our previous work \cite{lo, lm, ls} that most Ulrich vector bundles should be at least big, unless $X$ is covered by linear spaces of positive dimension. In particular any Ulrich vector bundle is very ample if $X$ does not contain lines by \cite[Thm.~1]{ls}. If $n \le 3$, the classification of non-big Ulrich vector bundles was achieved in \cite{lm}. On the other hand, when studying higher dimensional varieties, several new difficulties appear. For example, before this paper, as far as we know, no class of varieties of arbitrary dimension $n \ge 1$ for which non-big Ulrich vector bundles were classified was known except for $(\P^n, \O_{\P^n}(d))$ (in which case either $d=1$ and the only Ulrich vector bundles are $\O_{\P^n}^{\oplus r}, r \ge 1$ or $d \ge 2$ and they are all very ample by what we said above). 

A nice class of $n$-dimensional varieties, just coming after projective spaces, and for which Ulrich vector bundles are known, is the one of quadrics $Q_n \subset \P^{n+1}$. In fact, it follows by \cite[Rmk.~2.5(4)]{bgs} (see also \cite[Prop.~2.5]{b1}, \cite[Exa.~3.2]{ahmpl}) that the only indecomposable Ulrich vector bundles on $Q_n$ are the spinor bundles $\mathcal S, \mathcal S'$ and $\mathcal S''$ (see Definition \ref{spin}). Note that they are never ample as their restriction to lines is not ample by \cite[Cor.~1.6]{o}. On the other hand, it was not known which of these are big, unless $n \le 8$, this case following easily from \cite[Rmk.~2.9]{o}. 

Our first result is a classification of non-big Ulrich vector bundles on quadrics. 

\begin{bthm} 
\label{main1}

\hskip 3cm

An Ulrich vector bundle $\E$ on $Q_n$ is not big if and only if $\E$ is one of the following
$$\rm Table \ 1$$
$$\begin{tabular}{|c|c|c|c|}
\hline
$n$ & $\E$ \\
\hline
2 & $(\mathcal S')^{\oplus r}, (\mathcal S'')^{\oplus r}, r \ge1$ \\
\hline
3 & $\mathcal S$ \\
\hline
4 & $\mathcal S', \mathcal S'', \mathcal S' \oplus \mathcal S''$ \\
\hline
5 & $\mathcal S$ \\
\hline
6 & $\mathcal S', \mathcal S'', (\mathcal S')^{\oplus 2}, (\mathcal S'')^{\oplus 2}$ \\
\hline
10 & $\mathcal S', \mathcal S''$ \\
\hline
\end{tabular}.$$ 
\end{bthm}

The proof of the above theorem introduces, in fact on any variety $X$, a new geometrical estimate on the relation between Ulrich vector bundles and the Fano variety of some linear subspaces contained in $X$, see Proposition \ref{rette}. This estimate and the one given in Proposition \ref{fam1} are of independent interest and might have several applications. We show this by classifying Ulrich vector bundles of small numerical dimension $\nu(\E):=  \nu(\O_{\P(\E)}(1))$, a measure of positivity of $\E$. In general, in the presence of a rank $r$ Ulrich vector bundle $\E$, \cite[Thm.~2]{ls} shows that we can find two kinds of linear spaces covering $X$, namely the fibers of $\Phi_{\E} : X \to \mathbb G(r-1, \P H^0(\E))$ and the images on $X$ of the fibers of $\varphi_{\E} : \P(\E) \to \P H^0(\O_{\P(\E)}(1))$. Whenever the dimension of these spaces is at least $\frac{n}{2}$, Sato's classification \cite{sa2} can be applied. In \cite[Cor.~4]{ls} we dealt with the case of small numerical dimension of $\det \E$, corresponding to fibers of $\Phi_{\E}$. Instead, dealing with $\nu(\E)$, we have that $X$ is covered by linear spaces of dimension $n+r-1-\nu(\E)$, thus giving rise to the natural bound $\nu(\E) \le \frac{n}{2}+r-1$, appearing in the theorem below. We have:

\begin{bthm} 
\label{main2}

\hskip 3cm

Let $X \subseteq \P^N$ be a smooth irreducible variety of dimension $n \ge 1$ and let $\E$ be a rank $r$ Ulrich vector bundle on $X$. Then $\nu(\E) > \frac{n}{2}+r-1$ unless $(X,\O_X(1))$ is either:
\begin{itemize}
\item [(i)] $(\P(\F), \O_{\P(\F)}(1))$, where $\F$ is a rank $n-b+1$ very ample vector bundle over a smooth irreducible variety $B$ of dimension $b$ with $0 \le b \le \frac{n}{2}$.
\item [(ii)] $(Q_{2m}, \O_{Q_{2m}}(1))$ with $1 \le m \le 3$, $\nu(\E) = m+r-1$ and $\E \cong (\mathcal S')^{\oplus r}$ or $(\mathcal S'')^{\oplus r}$ when $m=1$, $\mathcal S'$ or $\mathcal S''$ when $2 \le m \le 3$. 
\end{itemize}
Moreover, in case (i) with $\nu(\E) \le \frac{n}{2}+r-1$, denoting by $p : X \cong \P(\F) \to B$ the projection map, we have two cases: 
\begin{itemize}
\item [(i1)] If $c_1(\E)^n=0$, then $\nu(\E) = b+r-1$ and $\E \cong p^*(\G(\det \F))$, where $\G$ is a rank $r$ vector bundle on $B$ such that $H^q(\G \otimes S^k \F^*)=0$ for $q \ge 0, 0 \le k \le b-1$ and $c_1(\G(\det \F))^b \ne 0$.
\item [(i2)] If $c_1(\E)^n>0$, then $b \le \frac{n}{2}-1, \nu(\E) \ge b+r$ and if $\nu(\E) = b+r$ then $\E_{|f}\cong T_{\P^{n-b}}(-1)\oplus \O_{\P^{n-b}}^{\oplus (r-n+b)}$ for any fiber $f=\P^{n-b}$ of $p$.
\end{itemize}
\end{bthm}

Note that the cases (i1) and (i2) actually occur at least in some special instances (for (i2) see Example \ref{un}, for (i1) pick any $\E \cong p^*(\G(\det \F))$ with $\G$ as in (i1), see \cite[Lemma 4.1(ii), Prop.'s 6.1 and 6.2]{lo}, \cite[Thm.~1]{lms} for some specific examples), while the cases in (ii) actually occur by Proposition \ref{big2}(ii).

It is usually difficult to compute the numerical dimension of Ulrich vector bundles, even for simple classes of varieties. Aside from projective spaces, perhaps the simplest ones are quadrics and linear $\P^k$-bundles. While for quadrics this can be now done, see Remark \ref{nd}, the same cannot be said for linear $\P^k$-bundles. In the latter case we have useful information when $k \le 2$ by Lemma \ref{nu=b+r}: there are only two cases: $\nu(\E)=n-k+r-1$ (and this is well-known) or $n-k+r$. A special but classical example with $k=2$ is the Del Pezzo fourfold of degree $6$. Another nice application of the methods in this paper is the classification of non-big Ulrich vector bundles on it, obtained using Lemma \ref{nu=b+r} and the resolution of Ulrich bundles given in \cite{ma}. 

\begin{bthm}
\label{p2xp2}

\hskip 3cm

Let $\P^2 \times \P^2 \subset \P^8$ be the Segre embedding and let $\E$ be a rank $r$ non-big Ulrich vector bundle on $\P^2 \times \P^2$. Then $\E \cong p^*(\O_{\P^2}(2))^{\oplus r}$, where $p: \P^2 \times \P^2 \to \P^2$ is one of the two projections.
\end{bthm}

Finally, we emphasize that Theorems \ref{main1}, \ref{main2} and \ref{p2xp2} will be important in our classification of non-big Ulrich vector bundles on fourfolds given in \cite{lms}.

\section{Notation and standard facts about (Ulrich) vector bundles}

Throughout this section we will let $X \subseteq \P^N$ be a smooth irreducible closed complex variety of dimension $n \ge 1$, degree $d$ and $H$ a hyperplane divisor on $X$.

\begin{defi}
For $k \in \Z : 1 \le k \le n$ we denote by $F_k(X)$ the Fano variety of $k$-dimensional linear subspaces of $\P^N$ that are contained in $X$. For $x \in X$, we denote by $F_k(X,x) \subset F_k(X)$ the subvariety of $k$-dimensional linear subspaces passing through $x$.
\end{defi}

\begin{defi}
Given a nef line bundle $\L$ on $X$ we denote by
$$\nu(\L) = \max\{k \ge 0: c_1(\L)^k \ne 0\}$$ 
the {\it numerical dimension} of $\L$. 
\end{defi}
As is well known (see for example \cite[(3.8)]{f}), when $\L$ is globally generated, $\nu(\L)$ is the dimension of the image of the morphism induced by $\L$.

\begin{defi}
\label{not}
Let $\E$ be a rank $r$ vector bundle on $X$. We denote by $c(\E)$ its Chern polynomial and by $s(\E)$ its Segre polynomial. We set $\P(\E) = \Proj(\Sym(\E))$ with projection map $\pi : \P(\E) \to X$ and tautological line bundle $\O_{\P(\E)}(1)$. We say that $\E$ is {\it nef (big, ample, very ample)} if $\O_{\P(\E)}(1)$ is nef (big, ample, very ample). If $\E$ is nef, we define the numerical dimension of $\E$ by $\nu(\E):=  \nu(\O_{\P(\E)}(1))$. When $\E$ is globally generated we define the map determined by $\E$ as
$$\Phi=\Phi_{\E} : X \to {\mathbb G}(r-1, \P H^0(\E))$$
and we set $\phi(\E)$ for the dimension of the general fiber of $\Phi_{\E}$. Moreover, we set
$$\varphi=\varphi_{\E} = \varphi_{\O_{\P(\E)}(1)} : \P(\E) \to \P H^0(\E)$$
$$\Pi_y = \pi(\varphi^{-1}(y)), y \in \varphi(\P(\E))$$
and  
$$P_x = \varphi(\P(\E_x)).$$
\end{defi}
Note that $\Phi(x)= [P_x]$ is the point in  ${\mathbb G}(r-1,\P H^0(\E))$ corresponding to $P_x$.

\begin{lemma}
\label{p}

Let $\E$ be a rank $r$ globally generated vector bundle on $X$. Let $x \in X$, so that $\P^{r-1} \cong P_x \subseteq \P H^0(\E)$. For any $y \in \varphi(\P(\E))$ we have that
$$y \in P_x \iff x \in \Pi_y.$$
\end{lemma}
\begin{proof}
Since $P_x = \varphi(\P(\E_x))$ and $\Pi_y = \pi(\varphi^{-1}(y))$, we have 
$$y \in P_x \iff \exists z \in \P(\E_x) \cap \varphi^{-1}(y) \iff \exists z \in \varphi^{-1}(y): \pi(z)=x \iff x \in \Pi_y.$$
\end{proof}

We recall the following well-known fact (see for example \cite[Prop.~10.2]{eh}). 

\begin{lemma}
\label{bf}
Let $\E$ be a rank $r$ globally generated vector bundle on $X$. Then
\begin{equation}
\label{num}
\nu(\E) = r-1+\max\{k\ge 0: s_k(\E)\ne 0\}.
\end{equation}
\end{lemma}

In order to check bigness of direct sums we will use the ensuing

\begin{lemma}
\label{big}
Let $\E, \F$ be two globally generated vector bundles on $X$. Then $s_i(\E^*)s_{n-i}(\F^*) \ge 0$ for all $0 \le i \le n$. Moreover $\E \oplus \F$ is big if and only if $s_i(\E^*)s_{n-i}(\F^*)>0$ for some $i \in \{0,\ldots,n\}$. In particular, if $\E$ is big then $\E \oplus \F$ is big.
\end{lemma}
\begin{proof}
Set $\xi:= \O_{\P(\E)}(1)$. First, note that the Segre classes $s_i(\E^*)$ and $s_i(\F^*)$ are effective and nef. In fact, $s_i(\E^*) = \pi_*\xi^{r-1+i}$ is effective because $\xi$ is globally generated. Also $s_i(\E^*)$ is nef because for every subvariety $Z \subseteq X$ of dimension $i$ we have that
$$s_i(\E^*) \cdot Z =  \xi^{r-1+i} \cdot  \pi^*Z \ge 0.$$
Therefore $s_i(\E^*)s_{n-i}(\F^*) \ge 0$ for all $0 \le i \le n$. Since $c(\E \oplus \F)=c(\E)c(\F)$ we get that 
$$s(\E)s(\F) c(\E \oplus \F) =s(\E)s(\F)c(\E)c(\F)=1$$ 
hence $s(\E \oplus \F)=s(\E)s(\F)$. It follows that
$$s_n((\E \oplus \F)^*)= \sum\limits_{i=0}^n s_i(\E^*)s_{n-i}(\F^*).$$
Hence $\E \oplus \F$ is big if and only if $s_n((\E \oplus \F)^*)>0$, if and only if $s_i(\E^*)s_{n-i}(\F^*)>0$ for some $i \in \{0,\ldots,n\}$. Also, if $\E$ is big then $s_n(\E^*)>0$, hence $\E \oplus \F$ is big.
\end{proof}

\begin{defi}
Let $\E$ be a vector bundle on $X \subseteq \P^N$. We say that $\E$ is an {\it Ulrich vector bundle} if $H^i(\E(-p))=0$ for all $i \ge 0$ and $1 \le p \le n$.
\end{defi}

The following properties will be often used without mentioning.

\begin{remark}
\label{gen}
Let $\E$ be a rank $r$ Ulrich vector bundle on $X \subseteq \P^N$ and let $d = \deg X$. Then
\begin{itemize}
\item [(i)] $\E$ is $0$-regular in the sense of Castelnuovo-Mumford, hence $\E$  is globally generated (by \cite[Thm.~1.8.5]{laz1}).
\item [(ii)] $\E_{|Y}$ is Ulrich on a smooth hyperplane section $Y$ of $X$ (by \cite[(3.4)]{b1}).
\end{itemize}
\end{remark}

\begin{remark}
\label{kno1}
On $(\P^n, \O_{\P^n}(1))$ the only rank $r$ Ulrich vector bundle is $\O_{\P^n}^{\oplus r}$ by \cite[Prop.~2.1]{es}, \cite[Thm.~2.3]{b1}. 
\end{remark}

\section{Ulrich bundles on quadrics}

For $n \ge 2$ we let $Q_n \subset \P^{n+1}$ be a smooth quadric. We let $S$ ($n$ odd), and $S', S''$ ($n$ even), be the vector bundles on $Q_n$, as defined in \cite[Def.~1.3]{o}.
 
\begin{defi}
\label{spin}
The {\it spinor bundles} on $Q_n$ are ${\mathcal S}={\mathcal S}_n = S(1)$ if $n$ is odd and ${\mathcal S}'={\mathcal S}'_n = S'(1)$, ${\mathcal S}''={\mathcal S}''_n = S''(1)$, if $n$ is even. They all have rank $2^{\lfloor \frac{n-1}{2} \rfloor}$.
\end{defi}

\begin{lemma}
\label{kno2}
With the above notation we have:
\begin{itemize}
\item [(i)] $s(\mathcal S)=c(S), s(\mathcal S')=c(S''), s(\mathcal S'')=c(S')$.
\item [(ii)] $\nu(\mathcal S) = 2^{\frac{n-1}{2}}-1+\max\{k \ge 0: c_k(S)\ne 0\}$, 

\noindent $\nu(\mathcal S') = \nu(\mathcal S'') = 2^{\frac{n-2}{2}}-1+ \max\{k \ge 0: c_k(S')\ne 0\}$.
\item [(iii)] The spinor bundles on $Q_n$ are Ulrich.
\item [(iv)] The only indecomposable Ulrich vector bundles on $Q_n$ are the spinor bundles.
\item [(v)] Spinor bundles are not ample.
\end{itemize}
\end{lemma}
\begin{proof}
By \cite[Thm.~2.8]{o} there are exact sequences 
\begin{equation}
\label{ott1}
0 \to S \to \O_Q^{\oplus 2^{\lfloor \frac{n+1}{2} \rfloor}} \to \mathcal S \to 0
\end{equation}
\begin{equation}
\label{ott2}
0 \to S' \to \O_Q^{\oplus 2^{\lfloor \frac{n+1}{2} \rfloor}} \to \mathcal S'' \to 0, \ \ \ 0 \to S'' \to \O_Q^{\oplus 2^{\lfloor \frac{n+1}{2} \rfloor}} \to \mathcal S' \to 0.
\end{equation}
From \eqref{ott1} we get $c(S)c(\mathcal S)=1$, hence $s(\mathcal S)=c(S)$. Similarly, from \eqref{ott2} we get that $s(\mathcal S'')=c(S')$ and $s(\mathcal S')=c(S'')$. This gives (i). Now (ii) follows by (i), \eqref{num} and the fact that if $n$ is even and $f$ is an automorphism of $Q_n$ exchanging the $\frac{n}{2}$-planes in $Q_n$, then $f^*S' \cong S''$ and $f^*S'' \cong S'$ (see \cite[page 304]{o}). (iii) follows by \cite[Thms.~2.3 and 2.8]{o}. (iv) follows by \cite{k} (see also \cite[Rmk.~2.5(4)]{bgs}, \cite[Prop.~2.5]{b1}, \cite[Exa.~3.2]{ahmpl}). As for (v), we know that $\mathcal S'_2 \cong \O_{\P^1}(1) \boxtimes \O_{\P^1}$ and $\mathcal S''_2 \cong \O_{\P^1} \boxtimes \O_{\P^1}(1)$, while for $n \ge 3$, \cite[Cor.~1.6]{o} gives that $\mathcal S_{|L}, \mathcal S'_{|L}, \mathcal S''_{|L}$ all decompose as $\O_{\P^1}^{\oplus 2^{\lfloor \frac{n-3}{2} \rfloor}} \oplus \O_{\P^1}(1)^{\oplus 2^{\lfloor \frac{n-3}{2} \rfloor}}$ for any line $L$ contained in $Q_n$, hence $\mathcal S, \mathcal S'$ and $\mathcal S''$ are not ample.
\end{proof}

We collect a preliminary result about the numerical dimension of spinor bundles.

\begin{prop}
\label{big2}
With the above notation we have:
\begin{itemize}
\item [(i)] $\mathcal S_n$ is big if and only if $c_n(S) \ne 0$, $\mathcal S'_n$ or $\mathcal S''_n$ is big if and only if $c_n(S') \ne 0$.
\item [(ii)] For $2 \le n \le 10$ we have:
$$\begin{tabular}{|c|c|c|c|}
\hline
$n$ & $\nu(\mathcal S)$ \\
\hline
3 & 3 \\
\hline
5 & 6 \\
\hline
7 & 14 \\
\hline
9 & 24 \\
\hline
\end{tabular} \ \ \
\begin{tabular}{|c|c|c|c|}
\hline
$n$ & $\nu(\mathcal S')=\nu(\mathcal S'')$ \\
\hline
2 & 1 \\
\hline
4 & 3 \\
\hline
6 & 6 \\
\hline
8 & 15 \\
\hline
10 & 24 \\
\hline
\end{tabular}$$
\item [(iii)] For $2 \le n \le 10$ the vector bundles in Table 1 of Theorem \ref{main1} are the only non-big Ulrich vector bundles $\E$ on $Q_n$.
\end{itemize}
\end{prop}
\begin{proof}
Using Lemma \ref{kno2}(i) we have that $\mathcal S$ is big if and only if
$$0 < s_n(\mathcal S^*)=(-1)^ns_n(\mathcal S)=(-1)^nc_n(S)$$ 
and we get (i) for $\mathcal S$. Similarly, (i) holds for $\mathcal S'$ and $\mathcal S''$. 

To see (ii), for $n \ne 7, 9, 10$, just use Lemma \ref{kno2}(ii) and \cite[Rmk.~2.9]{o}. If $n=7$ we have by \cite[Thm.~1.4]{o} that $S = (S'_8)_{|Q_7}$ and picking $H \in |\O_{Q_8}(1)|$ we deduce by \cite[Rmk.~2.9]{o} that
$$c_7(S)=c_7(S'_8) \cdot H=-H^8=-2.$$ 
Thus we get (ii) by Lemma \ref{kno2}(ii). 

If $n=10$, by \cite[Thm.~2.6(ii)]{o}, on any ${\mathbb P}^5$ in one of the families of $5$-planes in $Q_{10}$,  we have that

$$(S')_{|_{{\mathbb P}^5}}=\bigoplus_{i=0}^2 \Omega_{{\mathbb P}^5}^{2i}(2i), \qquad
(S'')_{|_{{\mathbb P}^5}}=\bigoplus_{i=0}^2 \Omega_{{\mathbb P}^5}^{2i+1}(2i+1)$$
and, in the other family
$$(S'')_{|_{{\mathbb P}^5}}=\bigoplus_{i=0}^2 \Omega_{{\mathbb P}^5}^{2i}(2i), \qquad
(S')_{|_{{\mathbb P}^5}}=\bigoplus_{i=0}^2 \Omega_{{\mathbb P}^5}^{2i+1}(2i+1).$$

This allows to compute the first five Chern classes of $S'$ and $S''$. In the  standard basis $\{e_0, \ldots, e_4, e_5, e_5', e_6, \ldots, e_{10}\}$ of the cohomology ring of $Q_{10}$ one gets that for $S'$ (respectively for $S''$): 
$$c_1=-8e_1, c_2=32e_2, c_3=-84e_3, c_4=160e_4, c_5=-244e_5-220e'_5 (\hbox{resp.} \ c_5= -220e_5-244e_5').$$
To get $c_i$, $i>5$, recall that any automorphism of $Q_{10}$ interchanging its two families of $5$-planes provides an isomorphism between $S'$ and $S''$, hence $c_i(S')=c_i(S'')$ for $i \ne 5$. 
From the exact sequence
$$0 \to S' \to \O_{Q_{10}}^{\oplus 2^5} \to S'(1) \to 0$$ 
and the isomorphism $(S')^* \cong S'(1)$  (see \cite[Thm.~2.8]{o}), we get the following relations in the Chern polynomials of $S'$ and $S''$:
\begin{equation}
\label{ott2appchern}
c(S')c((S')^*)=1
\end{equation}
and
\begin{equation}
\label{ott2appchern2}
c((S')^*)=c(S''(1)).
\end{equation}
By (\ref{ott2appchern}) one gets inductive formulae for the even Chern classes: 
$$c_{2k}=(-1)^{k+1}\frac{c_k^2}{2}+\sum_{i=1}^{k-1}(-1)^{k+1}c_ic_{2k-i}.$$
Respectively,  (\ref{ott2appchern2}) leads to inductive formulae for the odd ones:
$$c_{2k+1}=-\frac{1}{2}\sum_{i=0}^{2k}{16-i \choose 2k+1-i}c_ie_1^{2k+1-i}.$$ 
Hence, in coordinates in the standard basis quoted above, and recalling the relations
$e_1^i=2e_i$ for $6 \leq i \leq 10$, $e_5^2=(e_5')^2=0$ and $e_5e_5'=e_{10}$, the Chern classes of $S'$ (resp. $S''$) can be computed to give: 
\begin{equation}
\label{chern}
\begin{aligned}
& c_1=-8, c_2=32, c_3=-84, c_4=160, c_5=(-244,-220) \ \hbox{(resp.} \ c_5= (-220,-244)), \\
& c_6=528, c_7=-484, c_8=352, c_9=-176, c_{10}=0.
\end{aligned}
\end{equation}
If $n=9$, since $(S'_{10})_{|Q_9} \cong S$ by \cite[Thm.~1.4]{o}, we get that $c_9(S)=-176$, thus $\mathcal S$ is big and $\nu(\mathcal S)=24$.

To see (iii) note that Lemma \ref{kno2}(iv) gives that 
$$\E \cong (\mathcal S')^{\oplus a} \oplus (\mathcal S'')^{\oplus b} \ \hbox{when $n$ is even, with} \ a \ge 0, b \ge 0, a+b \ge 1$$ 
and
$$\E \cong \mathcal S^{\oplus a} \ \hbox{when $n$ is odd, with} \ a \ge 1.$$ 
If $n=2$ then $\mathcal S' = \O_{\P^1} \boxtimes \O_{\P^1}(1), \mathcal S''= \O_{\P^1}(1) \boxtimes \O_{\P^1}$ and $\det \E=a\mathcal S'+b\mathcal S'', c_2(\E)=ab$. Hence
$$s_2(\E^*)=c_1(\E)^2-c_2(\E)=(a\mathcal S'+b\mathcal S'')^2-ab=ab.$$
Thus $\E$ is not big if and only if $s_2(\E^*)=0$, that is if and only if $a=0$ or $b=0$. 

If $n=3$ it is proved in \cite[Rmk.~2.7]{lm} that $\E$ is not big if and only if $a=1$. 

If $n = 4$  we need to show that $\E$ is not big if and only if $(a,b) \in \{(0,1),(1,0),(1,1)\}$. We know by (ii) that $\mathcal S'$ and $\mathcal S''$ are not big since their numerical dimension is less than $n+r-1=5$. Consider now $\E = \mathcal S' \oplus \mathcal S''$. By Lemma \ref{big} and Lemma \ref{kno2}(i) we need to prove that $c_i(S')c_{4-i}(S'')=0$ for all $0 \le i \le 4$. Since they have rank $2$ we just need the case $i=2$. By \cite[Rmk.~2.9]{o} we know that $c_2(S')=\ell, c_2(S'')=h^2-\ell$ where $h$ is the class of a hyperplane section and $\ell$ of a plane in $Q_4$, with $h^2\ell=\ell^2=1$. Hence $c_2(S')c_2(S'')=\ell(h^2-\ell)=0$ and $\E$ is not big. Vice versa, to conclude the case $n=4$, using Lemma \ref{big}, it remains to show that $(\mathcal S')^{\oplus 2}$ and $(\mathcal S'')^{\oplus 2}$ are big. This follows as above since $c_2(S')^2=\ell^2=1$ and $c_2(S'')^2=(h^2-\ell)^2=1$.

If $n = 5$ we know by (ii) that $\mathcal S$ is not big. Proceeding as above, we need to show that $\mathcal S^{\oplus 2}$ is big. Now \cite[Rmk.~2.9]{o} gives that $c_2(S)c_3(S)=(2h^2)(-h^3)=-4 \ne 0$, hence $\mathcal S^{\oplus 2}$ is big.

If $n = 6$  we need to show that $\E$ is not big if and only if $(a,b) \in \{(0,1),(1,0),(0,2), (2,0)\}$. To see that these are not big note that the case $(2,0)$ implies the case $(1,0)$ and the case $(0,2)$ implies the case $(0,1)$ by Lemma \ref{big}. Consider $\E = (\mathcal S')^{\oplus 2}$ or $(\mathcal S'')^{\oplus 2}$. By Lemma \ref{big} and Lemma \ref{kno2}(i) we need to prove that $c_i(S'')c_{6-i}(S'')=0$ or $c_i(S')c_{6-i}(S')=0$, for all $0 \le i \le 6$. On the other hand \cite[Rmk.~2.9]{o} gives that $c_i(S')=c_i(S'')=0$ for all $4 \le i \le 6$ and $c_3(S')=-2\ell, c_3(S'')=-2(h^3-\ell)$, where $h$ is the class of a hyperplane section and $\ell$ of a $3$-plane in $Q_6$, with $h^3\ell=1, \ell^2=0$. Hence $c_3(S'')^2=4\ell^2=0, c_3(S'')^2=4(h^3-\ell)^2=0$ and $\E$ is not big. Vice versa, to conclude the case $n=6$, using Lemma \ref{big}, it remains to show that $\mathcal S' \oplus \mathcal S'',  (\mathcal S')^{\oplus 3}$ and $(\mathcal S'')^{\oplus 3}$ are big. Since $c_3(S')c_3(S'')=4\ell (h^3-\ell)=4$ we get, as above, that $\mathcal S' \oplus \mathcal S''$ is big. Now Lemma \ref{kno2}(i) gives
$$s_3((\mathcal S')^{\oplus 2})s_3(\mathcal S')=2(s_3(\mathcal S')+s_1(\mathcal S')s_2(\mathcal S'))c_3(S'')=-8(\ell-3h^3)(h^3-\ell)=16$$
and Lemma \ref{big} implies that $(\mathcal S')^{\oplus 3}$ is big. Similarly, $(\mathcal S'')^{\oplus 3}$ is big.

If $n=7, 8, 9$ the spinor bundles are big by (ii), hence so is $\E$ by Lemma \ref{big}. 

Finally, if $n=10$ (see \eqref{chern}) we have the following values of the Chern classes of $S'$ and $S''$:
$$c_1=-8, c_9=-176 \ \hbox{and} \ c_{10}=0,$$
thus $\mathcal S'$ and $\mathcal S''$ are not big by (i). On the other hand, $c_1(S')c_9(S')=c_1(S'')c_9(S'')=c_1(S')c_9(S'')=c_1(S'')c_9(S')=1408 \ne 0$ so that $\E$ is big if $(a,b) \not\in \{(1,0), (0,1)\}$ by Lemma \ref{big} and Lemma \ref{kno2}(i).
\end{proof}

\section{Behaviour of Ulrich bundles on linear subspaces}

We start by analyzing the behaviour on lines.

\begin{prop} 
\label{rette}
Let $X \subseteq \P^N$ be a smooth irreducible variety of dimension $n \ge 2$. Let $\E$ be a non-big rank $r$ Ulrich vector bundle on $X$. Let $x \in X$ and let
$$h(\E,x) = \max\{h \ge 1 : \exists L \in F_1(X,x) \ \hbox{such that} \ \O_{\P^1}^{\oplus h}  \ \hbox{is a direct summand of} \ \E_{|L}\}.$$
Then
\begin{equation}
\label{stima}
\dim F_1(X,x) + h(\E,x) \ge r.
\end{equation}
Moreover, if $h(\E) = \max\{h(\E,x), x \in X\}$, then
\begin{equation}
\label{stima2}
\nu(\E) \le \dim F_1(X) + h(\E)-1.
\end{equation}
\end{prop}
\begin{proof}
Set $\varphi = \varphi_{\E}$. We will use the fact, following by \cite[Lemma 3.3]{ls}, that for any line $L \subset X$, we have that $H^0(\E) \to H^0(\E_{|L})$ is surjective, hence the restriction of $\varphi$ to $\P(\E_{|L})$ is the tautological morphism $\varphi_{(\E_{|L})}$ associated to $\O_{\P(\E_{|L})}(1)$.

For any $L \in F_1(X,x)$ we can write
$$\E_{|L} \cong \O_{\P^1}^{\oplus h_L} \oplus \O_{\P^1}(a_{1,L}) \oplus \ldots \oplus \O_{\P^1}(a_{r-h_L,L})$$
for some integers $0 \le h_L \le r$ and $a_{i,L} \ge 1$ for $1 \le i \le r-h_L$. Suppose that $h_L \ge 1$. Then $\varphi(\P(\E_{|L})) \subseteq \P H^0(\E_{|L})$ is either a rational normal scroll with vertex $V_L \cong \P^{h_L-1}$ when $h_L \le r-1$ (including the case $h_L=r-1, a_{1,L}=1$, in which $\varphi(\P(\E_{|L}))=\P H^0(\E_{|L})$) or $\varphi(\P(\E_{|L}))=\P^{r-1}$ when $h_L = r$. In the latter case we define $V_L = \varphi(\P(\E_{|L}))$. Moreover $\P(\O_{\P^1}^{\oplus h_L}) \subseteq \P(\E_{|L})$ and if we call $\psi_L : \P(\O_{\P^1}^{\oplus h_L}) \to V_L$ the map associated to the tautological line bundle on $\P(\O_{\P^1}^{\oplus h_L})$, the fibers of $\psi_L$ are the curves $\psi_L^{-1}(y), y \in V_L$. Note that they all intersect $\P(\E_x)$ in one point.

Consider the incidence correspondence
$$\I = \{ (z, L) \in \P(\E_x) \times F_1(X,x) : h_L \ge 1 \ \hbox{and} \ \exists y \in V_L \ \hbox{with} \ z \in \psi_L^{-1}(y) \cap \P(\E_x) \}$$
together with its projections 
$$\xymatrix{& \I \ \ \ar[dl]_{p_1} \ar[dr]^{p_2} & \\ \ \ \ \ \ \P(\E_x) & & \ F_1(X,x) \ \ \ .}$$
We claim that $p_1$ is surjective. 
 
Let $z \in \P(\E_x)$, so that $x = \pi(z)$. Let $y=\varphi(z)$. Since $\E$ is not big, we have that $\dim \varphi(\P(\E)) < \dim \P(\E)$, hence all fibers of $\varphi$ are positive dimensional. Thus $\dim \varphi^{-1}(y) > 0$, hence we can find $z' \in \varphi^{-1}(y)$ with $z' \ne z$. Let $x' = \pi(z')$. Note that $x \ne x'$, because otherwise $z, z' \in \P(\E_x)$, contradicting the fact that $\varphi_{|\P(\E_x)}$ is an embedding. Consider the line $L = \langle x,x' \rangle$. If $L \not\subset X$, then \cite[Lemma 3.2]{ls} implies that $P_x \cap P_{x'} = \emptyset$. On the other hand, since $z, z' \in \varphi^{-1}(y)$, we have that $x, x' \in \Pi_y$. Therefore Lemma \ref{p} implies that $y \in P_x \cap P_{x'}$, a contradiction. Thus we have proved that $L \subset X$, that is $L \in F_1(X,x)$. 

Since $z$ and $z'$ are not separated by $\varphi_{(\E_{|L})}$, we have that $h_L \ge 1, z \in \P(\O_{\P^1}^{\oplus h_L})$ and $y=\varphi(z)=\psi_L(z) \in V_L$. Therefore $z \in \psi_L^{-1}(y) \cap \P(\E_x)$ and then $(z, L) \in \I$, giving that $z \in \Im p_1$.  Hence $p_1$ is surjective, $h(\E,x)$ is well defined and let us see that
\begin{equation}
\label{vert}
\varphi(\P(\E))=\bigcup\limits_{L \in F_1(X): h_L \ge 1} V_L.
\end{equation}
In fact, on one side, the inclusion $V_L \subseteq \varphi(\P(\E))$ follows by definition. Now let $y \in \varphi(\P(\E))$. Then there is $z \in \P(\E)$ such that $y=\varphi(z)$ and set $x = \pi(z)$. Then $z \in \P(\E_x)$ and the surjectivity of $p_1$ implies that there is a line $L \in F_1(X,x)$ such that $(z, L) \in \I$, hence $h_L \ge 1$ and $y = \varphi(z) = \psi_L(z) \in V_L$. This proves \eqref{vert}.

Now observe that the nonempty fibers of $p_2$ are all isomorphic to $V_L$ for some $L \in F_1(X,x)$. In fact, if $L \in F_1(X,x)$ is such that $p_2^{-1}(L) \ne \emptyset$, then there is $z_0 \in \P(\E_x)$ such that $(z_0, L) \in \I$, hence $h_L \ge 1$. Then, for any $(z, L) \in p_2^{-1}(L)$, there exists $y \in V_L$ with $z \in \psi_L^{-1}(y) \cap \P(\E_x)$, that is $\psi_L(z)=y \in V_L$. This defines a morphism $f: p_2^{-1}(L) \to V_L$ by $f((z, L))=\psi_L(z)$. We have that $f$ is injective since all the fibers $\psi_L^{-1}(y), y \in V_L$ intersect $\P(\E_x)$ in one point. Also, if $y \in V_L$, let $\{z\}=\psi_L^{-1}(y) \cap \P(\E_x)$. Then $(z, L) \in p_2^{-1}(L)$ and $y=\psi_L(z)$. Thus $f$ is also surjective.

Let $W$ be an irreducible component of $\I$ such that $W$ dominates $\P(\E_x)$. We have, for a general $L \in p_2(W)$,
$$\dim W = \dim p_2(W)+ h_L - 1 \le \dim F_1(X,x) + h(\E,x)-1.$$ Therefore we deduce that
$$r-1 = \dim \P(\E_x) \le \dim W \le \dim F_1(X,x) + h(\E,x)-1$$
and \eqref{stima} holds. Now consider the incidence correspondence
$$\I' = \{ (y, L) \in \P H^0(\E) \times F_1(X) : h_L \ge 1 \ \hbox{and} \ y \in V_L\}$$
together with its projections 
$$\xymatrix{& \I' \ \ \ar[dl]_{p_1'} \ar[dr]^{p_2'} & \\ \ \ \ \ \ \P H^0(\E) & & \ F_1(X) \ \ \ .}$$
Then $\Im p_1' = \varphi(\P(\E))$ by \eqref{vert}. Also, the nonempty fibers of $p_2'$ are isomorphic to $V_L \cong \P^{h_L-1}$, since if $(p_2')^{-1}(L) \ne \emptyset$ then $h_L \ge 1$ and $(p_2')^{-1}(L)=\{(y, L): y \in V_L\} \cong V_L$. Therefore, choosing an irreducible component $W'$ of $\I'$ such that $W'$ dominates $\varphi(\P(\E))$, we get, for a general $L \in p_2(W')$, that
$$\nu(\E) = \dim \varphi(\P(\E)) \le \dim W' \le \dim F_1(X) + h_L-1 \le \dim F_1(X) + h(\E)-1$$
and we get \eqref{stima2}.
\end{proof}

We can now prove our first theorem.

\renewcommand{\proofname}{Proof of Theorem \ref{main1}}

\begin{proof}
The vector bundles in Table 1 are Ulrich and non-big by Proposition \ref{big2}(iii). Vice versa, by the same proposition, we can assume that $n \ge 11$. Let $\E$ be a spinor bundle on $Q_n$. Note that $r = 2^{\lfloor \frac{n-1}{2} \rfloor}$ and $h(\E,x)=2^{\lfloor \frac{n-3}{2} \rfloor}$ by \cite[Cor.~1.6]{o} for any $x \in Q_n$. We claim that $\E$ is big. In fact, if not, we get by Proposition \ref{rette} that 
$$n - 2 + 2^{\lfloor \frac{n-3}{2} \rfloor} \ge 2^{\lfloor \frac{n-1}{2} \rfloor}$$
contradicting $n \ge 11$. Now just apply Lemma \ref{big} together with Lemma \ref{kno2}(iv) to get that any Ulrich vector bundle on $Q_n, n \ge 11$ is big.
\end{proof}
\renewcommand{\proofname}{Proof}

\begin{remark}
\label{nd}
It follows by Theorem \ref{main1} that the numerical dimension of any rank $r$ Ulrich vector bundle $\E$ on $Q_n$ is known. In fact, if $\E$ does not belong to Table 1, then $\nu(\E)=n+r-1$. Now suppose that $\E$ is as in Table 1. Then $\nu(\E)$ is given in Proposition \ref{big2}(ii) if $\E$ is a spinor, $\nu(\E)=n+r-2$ if $\E$ is not a spinor.
\end{remark}

\begin{remark}
The example $\E = \mathcal S' \oplus \mathcal S''$ on $Q_4$ shows that, even when $\Pic(X) \cong \Z$, the fibers of $\varphi$ can have different dimensions. In fact, $\nu(\E)=6$ so that a general fiber is $1$-dimensional, while $\nu(\mathcal S')=3$ hence the fibers over points in $\P(\mathcal S')$ are $2$-dimensional.
\end{remark}

Under a suitable hypothesis, we will now study Ulrich vector bundles with minimal numerical dimension on linear $\P^k$-bundles and determine their restriction to fibers in the next case. We will use the notation in Definition \ref{not}. 

The following lemma allows us to identify $\E_{|f}$ for all $f$, rather than on a general $f$, thus giving a better description of $\E$. This is needed, in the present paper, in Theorem \ref{main2}(i2) and in the proof of \cite[Thm.~1]{lms} (for example in case (xii)).

\begin{lemma}
\label{nu=b+r}
Let $(X,\O_X(1))=(\P(\F), \O_{\P(\F)}(1))$, where $\F$ is a rank $n-b+1$ very ample vector bundle over a smooth irreducible variety $B$ of dimension $b$ with $1 \le b \le n-1$. Let $\E$ be a rank $r$ Ulrich vector bundle on $X$, let $p : X \to B$ be the projection morphism and suppose that
\begin{equation}
\label{cos}
p_{|\Pi_y} : \Pi_y \to B \ \hbox{is constant for every} \ y \in \varphi(\P(\E)).
\end{equation}
Then, for every fiber $f$ of $p$ we have $\nu(\E)=b+\dim \varphi(\pi^{-1}(f)) \ge b+r-1$. Moreover we have the following two extremal cases:
\begin{itemize}
\item [(i)] If $\nu(\E) = b+r-1$ there is a rank $r$ vector bundle $\G$ on $B$ such that $\E \cong p^*(\G(\det \F))$ and
$H^j(\G \otimes S^k \F^*)=0$ for all $j \ge 0, 0 \le k \le b-1$. 
\item [(ii)] If $\nu(\E) = b+r$ then either $b=n-1$ and $\E$ is big or $b \le n-2$ and $\E_{|f} \cong T_{\P^{n-b}}(-1)\oplus \O_{\P^{n-b}}^{\oplus (r-n+b)}$ for any fiber $f=\P^{n-b}$ of $p$.
\end{itemize}
\end{lemma}
\begin{proof}
Set $f_v = p^{-1}(v), v \in B$. Let $x, x'\in X$ be such that $p(x) \ne p(x')$. If there exists an $y \in P_x \cap P_{x'}$, then $x, x' \in \Pi_y$ by Lemma \ref{p}, contradicting \eqref{cos}. Therefore $P_x \cap P_{x'} = \emptyset$.
It follows that 
\begin{equation}
\label{disg}
\varphi(\P(\E))=\bigsqcup_{v \in B}\varphi(\pi^{-1}(f_v)).
\end{equation}
Consider the incidence correspondence
$$\I=\{(y, v) \in \varphi(\P(\E)) \times B : y \in \varphi(\pi^{-1}(f_v))\}$$
together with its two projections $p_1 : \I \to \varphi(\P(\E))$ and $p_2 : \I \to B$. Then \eqref{disg} implies that $p_1$ is bijective, hence $\I$ is irreducible and $\dim \I = \dim \varphi(\P(\E)) = \nu(\E)$. Since $p_2$ is surjective, it follows that for any $v \in B$ we have that 
$$\dim \varphi(\pi^{-1}(f_v)) \ge \nu(\E)-b.$$ 
On the other hand, for every $y \in \varphi(\pi^{-1}(f_v))$ there is a $z \in \pi^{-1}(f_v)$ such that $y=\varphi(z)$, hence $\pi(z) \in \Pi_y \cap f_v$ and \eqref{cos} gives that $p(\Pi_y)=\{v\}$. Hence $\Pi_y \subseteq f_v$ and therefore $\varphi^{-1}(y) \subseteq \pi^{-1}(f_v)$. Picking a general $y \in \varphi(\pi^{-1}(f_v))$ we deduce that
$$\dim \varphi(\pi^{-1}(f_v)) = \dim \pi^{-1}(f_v) - \dim \varphi^{-1}(y) \le n-b+r-1 - (n+r-1-\nu(\E))=\nu(\E)-b.$$
Therefore $\dim \varphi(\pi^{-1}(f)) = \nu(\E)-b$ for every fiber $f$ of $p : X \to B$. In particular, being $\varphi(\pi^{-1}(f_v))$ union of linear spaces $P_x=\P^{r-1}$ for $x \in f_v$, we see that $\nu(\E)\ge b+r-1$.

Now consider the morphism $\Phi_{|f} : f=\P^{n-b} \to  {\mathbb G}(r-1, \P H^0(\E))$. Observe that $\Phi_{|f} $ is constant if and only if $P_x=P_{x'}$ for any $x, x' \in f$, that is if and only if $\varphi(\pi^{-1}(f)) = P_x = \P^{r-1}$, or, equivalently, if and only if $\nu(\E) = b+r-1$. Hence, when $\nu(\E) = b+r-1$, we get that $\det(\E)_{|f}$ is trivial. This implies that there is a globally generated line bundle $M$ on $B$ such that $\det \E = p^*M$ and therefore $\E$ is as in (i) by \cite[Lemmas 5.1 and 4.1]{lo}. 

On the other hand, if $\nu(\E) = b+r$, then $\Phi_{|f}$ is finite-to-one onto its image. If $b=n-1$ then $\nu(\E) = n+r-1$ and $\E$ is big. If $b \le n-2$ then $\varphi(\pi^{-1}(f)) \subseteq \P H^0(\E)$ is swept out by a family $\{P_x, x \in f\}$ of dimension $n-b \ge 2$ of linear $\P^{r-1}$'s. Therefore
\begin{equation}
\label{pro}
\varphi(\pi^{-1}(f))=\P^r \ \hbox{for every fiber} \ f \ \hbox{of} \ p:X\to B.
\end{equation}
This gives that $\E_{|f}$ can be generated by $r+1$ global sections and we get an exact sequence 
\begin{equation}
\label{ex}
0 \to \O_f(-a) \to \O_f^{\oplus (r+1)} \to \E_{|f}\to 0
\end{equation} 
where $c_1(\E_{|f})=\O_f(a)$. Note that $H^1(\O_f(-a-1))=0$ since $f=\P^{n-b}, n-b \ge 2$ and then \eqref{ex} implies that 
\begin{equation}
\label{ex2}
H^0(\E_{|f}(-1))=0. 
\end{equation} 
Since $\varphi(\pi^{-1}(f))=\P^r$, there must be two points $x, x' \in f$ such that $P_x \ne P_{x'}$ and \eqref{pro} gives that $\varphi(\pi^{-1}(L))=\P^r$, where $L$ is the line in $f$ joining $x$ and $x'$. Now \cite[Lemma 3.2]{ls} gives that $H^0(\E) \to H^0(\E_{|L})$ is surjective, hence $\varphi_{|\pi^{-1}(L)} = \varphi_{\E_{|L}}$. Since $c_1(\E_{|L})=\O_L(a)$ we get that
$$a = \deg \varphi_{\E_{|L}}(\pi^{-1}(L)) = \deg \varphi(\pi^{-1}(L))= \deg \P^r=1$$
where the degrees are meant as subvarieties of $\P H^0(\E)$.

Now for any line $L' \subset f$ we have that $c_1(\E_{|L'})=\O_{L'}(1)$, hence, being $\E$ globally generated, we get that $(\E_{|f})_{|L'} \cong \O_{\P^1}(1) \oplus \O_{\P^1}^{\oplus (r-1)}$. Then \cite[Prop.~IV.2.2]{e} implies that $\E_{|f} \cong T_{\P^{n-b}}(-1) \oplus \O_{\P^{n-b}}^{\oplus (r-n+b)}$ or $\O_{\P^{n-b}}(1) \oplus \O_{\P^{n-b}}^{\oplus (r-1)}$ or $\Omega_{\P^{n-b}}(2) \oplus \O_{\P^{n-b}}(1)^{\oplus (r-n+b)}$. Now the second case is excluded by \eqref{ex2}. On the other hand, if the third case occurs, then $\O_f(1) = c_1(\E_{|f})=\O_f(r-1)$, hence $2 = r \ge n-b \ge 2$, therefore equality holds, and this is also the first case. Thus $\E$ is as in (ii).
\end{proof}

We now give an example showing that the restriction of $\E$ in Lemma \ref{nu=b+r}(ii) actually occurs. For an example with $b = 2$ see \cite[Ex.~5.11]{lms}.

\begin{example}
\label{un}

Let $n \ge 3$, let $X = \P^1 \times \P^{n-1}$ and let $\O_X(1)=\O_{\P^1}(1) \boxtimes \O_{\P^{n-1}}(1)$. It is easily seen that, for every $r \ge n-1$, the vector bundle
$$\E =  [\O_{\P^1}(n-2) \boxtimes T_{\P^{n-1}}(-1)] \oplus  [\O_{\P^1}(n-1) \boxtimes \O_{\P^{n-1}}]^{\oplus (r-n+1)}$$
is Ulrich, $c_1(\E)^n > 0, \nu(\E)=r+1$ and $\E_{|f} = T_{\P^{n-1}}(-1) \oplus  \O_{\P^{n-1}}^{\oplus (r-n+1)}$ of any fiber $f$ of the first projection $p: X \to \P^1$.
\end{example}

A first use of Lemma \ref{nu=b+r} is the following.

\renewcommand{\proofname}{Proof of Theorem \ref{p2xp2}}
\begin{proof}
If $c_1(\E)^4=0$ we have by \cite[Cor.~4.9]{ls} that $\E \cong p^*(\O_{\P^2}(2))^{\oplus r}$, where $p : \P^2 \times \P^2 \to \P^2$ is one of the two projections. 

Assume now that $c_1(\E)^4>0$. First, we claim that we can apply Lemma \ref{nu=b+r}(ii) with one of the two projections $p : \P^2 \times \P^2 \to \P^2$ or $q : \P^2 \times \P^2 \to \P^2$. In fact, \cite[Cor.~2]{ls} implies that $r+3-\nu(\E) = 1$,
so that $\nu(\E) = r+2$. Also, for any $y \in \varphi(\P(\E))$, we have that $\Pi_y$ is a linear subspace in $\P^8$ of dimension
$k \ge r+3-\nu(\E)=1$, hence, as is well known, $\Pi_y$ is contained in a fiber of $p$ or in a fiber of $q$. Let $U \subseteq \varphi(\P(\E))$ be the non-empty open subset on which $\dim \varphi^{-1}(y)=1$ for every $y \in U$. This gives a morphism $\gamma : U \to F_1(X)$ defined by $\gamma(y)=\Pi_y$. Since $F_1(X)$ has two irreducible disjoint components, namely $W_p$, the lines contained in a fiber of $p$ and $W_q$ the lines contained in a fiber of $q$, we get that either $\gamma(U) \subseteq W_p$ or $\gamma(U) \subseteq W_q$. In the first case, by specialization, we deduce that $\Pi_y$ is contained in a fiber of $p$ for every $y \in \varphi(\P(\E))$. In the second case the same happens for $q$.
Therefore we can assume that \eqref{cos} holds for $p$. Then Lemma \ref{nu=b+r}(ii) applies to $p$ and we get that 
\begin{equation}
\label{restr}
\E_{|f} \cong T_{\P^2}(-1)\oplus \O_{\P^2}^{\oplus (r-2)}
\end{equation} 
on any fiber $f \cong \P^2$ of $p$. Let $A = p^*(\O_{\P^2}(1))$ and  $B = q^*(\O_{\P^2}(1))$, so that we can write 
$$\det \E = \alpha A + \beta B, c_2(\E) = \gamma A^2 + \delta A \cdot B + \epsilon B^2$$
for some $\alpha, \beta, \gamma, \delta, \epsilon \in \Z$. Now \eqref{restr} implies that $\alpha=\gamma=1$. Moreover, if $C$ is the curve section of $X \subset \P^8$, then $C \subset \P^5$ is an elliptic curve of degree $6$ and $\E_{|C}$ is a rank $r$ Ulrich vector bundle on $C$ by Remark \ref{gen}(ii). Hence $\chi(\E_{|C}(-1))=0$, that is $\deg(\E_{|C})=6r$ and we get
$$3 + 3 \beta = (A+\beta B)(A+B)^3 = c_1(\E_{|C}) = 6r$$
that is $\beta=2r-1$ and this gives
\begin{equation}
\label{classi}
\det \E = A + (2r-1) B, c_2(\E) = A^2 + \delta AB + \epsilon B^2.
\end{equation}
Let now $Y$ be the hyperplane section of $X$, so that $\E_{|Y}$ is a rank $r$ Ulrich vector bundle on $Y$ by Remark \ref{gen}(ii). By \cite[Thm.~5.1]{ma}, setting $p_1 = p_{|Y} : Y \to \P^2, p_2 = q_{|Y} : Y \to \P^2$ and $\mathcal G_i = p_i^*(\Omega_{\P^2}(1)), i=1,2$, there is a resolution
\begin{equation}
\label{riso}
0 \to p_1^*(\O_{\P^2}(1))^{\oplus d} \oplus p_2^*(\O_{\P^2}(1))^{\oplus c} \to \mathcal G_1(1)^{\oplus b} \oplus \mathcal G_2(1)^{\oplus a} \to \E_{|Y} \to 0.
\end{equation}
Computing rank and the first Chern class in \eqref{riso} and using \eqref{classi} we get the three equations
$$2a+2b=c+d+r, 1=b+2a-d, 2r-1=2b+a-c$$
which imply that $a+c=1$, thus giving the only possible solutions $(a,b,c,d)=(0,r,1,r-1)$ or $(1,r-1,0,r)$. Now computing the second Chern class in \eqref{riso} and using \eqref{classi} we have two possibilities. In the first case we get that $r=1$, a contradiction since $\E$ is not big. In the second case we get that $r=2, \det \E = A + 3B$ and $c_2(\E) = A^2 + AB + 4B^2$, but then $s_4(\E^*)=6>0$, a contradiction since $\E$ is not big.
\end{proof}
\renewcommand{\proofname}{Proof}

Next we prove a useful result that will later allow to give an upper bound on the rank. We will use the notation in Definition \ref{not}. 

\begin{prop} 
\label{fam1}
Let $X \subseteq \P^N$ be a smooth irreducible variety of dimension $n \ge 2$ and let $\E$ be a rank $r$ Ulrich vector bundle on $X$ such that $\E$ is not big and $c_1(\E)^n>0$. Assume that for a general point $x \in X$ we have that $F_{n+r-\nu(\E)}(X,x) = \emptyset$. Then there is a morphism, finite onto its image, 
$$\psi: \P^{r-1} \cong P_x \to F_{n+r-1-\nu(\E)}(X,x)$$ 
and
$$\dim F_{n+r-1-\nu(\E)}(X,x) \ge r-1.$$
\end{prop}
\begin{proof}
By \cite[Thm.~2]{ls} and Lemma \ref{p}, for any $y \in P_x$ we have that $\Pi_y$ is a linear subspace in $\P^N$ of dimension $k \ge n+r-1-\nu(\E) \ge 1$ such that $x \in \Pi_y \subseteq X$. Since $F_{n+r-\nu(\E)}(X,x) = \emptyset$ we get that $k = n+r-1-\nu(\E)$ and we can define a morphism
$$\psi: \P^{r-1} \cong P_x \to F_k(X,x)$$
by $\psi(y)=\Pi_y$. If $\psi$ is constant, we claim that for any $y \in P_x$ we have that
$$\Pi_y \subseteq \Phi^{-1}(\Phi(x)).$$
In fact, let $x' \in \Pi_y$. For every $y' \in P_x$, since $\Pi_y=\Pi_{y'}$, we have that $x' \in \Pi_{y'}$, that is $y' \in P_{x'}$ by Lemma  \ref{p}. Hence $P_x \subseteq P_{x'}$ and we deduce that $P_x = P_{x'}$ and therefore 
$$\Phi(x)=[P_x]=[P_{x'}]=\Phi(x')$$
and the claim is proved. Now $\Pi_y$ has dimension $k \ge 1$, hence the fibers of $\Phi$ have dimension at least $1$. This implies that $\det \E$ is not big, a contradiction. Therefore $\psi$ is finite onto its image and we deduce that $\dim F_k(X,x) \ge r-1$.
\end{proof}

We can now prove our second theorem.

\renewcommand{\proofname}{Proof of Theorem \ref{main2}}

\begin{proof}
Suppose that $\nu(\E) \le \frac{n}{2}+r-1$. 

If $n=1$ then $\nu(\E) \le r-1$. But $\varphi(\P(\E_x)) = P_x = \P^{r-1}$, hence $\nu(\E) = r-1$ and $P_x = P_{x'}$ for every $x \ne x' \in X$, that is $\Phi(x)=\Phi(x')$. Hence \cite[Thm.~2]{ls} gives that $(X,\O_X(1)) \cong (\P^1,\O_{\P^1}(1))$ and we are in case (i). Vice versa, if we are in case (i) then $b=0$ and $(X,\O_X(1)) \cong (\P^1,\O_{\P^1}(1))$. It follows by Remark \ref{kno1} that $\E \cong \O_{\P^1}^{\oplus r}$, hence $\nu(\E)=r-1$.

Suppose from now on that $n \ge 2$. 

Let $x \in X$. By \cite[Thm.~2]{ls}, for any $y \in P_x$ we have that $\Pi_y$ is a linear subspace of dimension $k$ in $\P^N$ such that, using Lemma \ref{p}, $x \in \Pi_y \subseteq X$ and $k \ge n+r-1-\nu(\E) \ge \frac{n}{2}$. 

Assume that we are not in case (i). 

Then \cite[Main Thm.]{sa2} implies that $(X,H)$ is one of the following:
\begin{itemize}
\item [(1)] $(Q_{2m}, \O_{Q_{2m}}(1))$.
\item [(2)] $(\mathbb G(1,m+1), \O_{\mathbb G(1,m+1)}(1))$ (the Pl\"ucker line bundle).
\end{itemize}
We can assume that $m \ge 2$, for otherwise, when $m=1$, case (2) is also case (i) and in case (1) we know that $\E \cong (\mathcal S')^{\oplus r}$ or $(\mathcal S'')^{\oplus r}$ by Proposition \ref{big2}(iii), thus giving case (ii). Note that $\det \E$ is globally generated and big by \cite[Lemma 3.2]{lo}, since we are excluding (i). In particular $r \ge 2$. Now $n=2m$ and $F_{m+1}(X,x)= \emptyset$. Hence $m \ge k \ge 2m+r-1-\nu(\E) \ge m$ and we get that $k = m$ and $\nu(\E)=m+r-1$. Then $\dim F_m(X,x) \ge r-1$ by Proposition \ref{fam1}. In case (1) we know by Lemma \ref{kno2}(iv) that $r \ge 2^{m-1}$ and $\dim F_m(X,x)= \frac{m(m-1)}{2}$, thus the only possibilities are for $2 \le m \le 3$ and $r=2^{m-1}$. Hence $\E$ is a spinor bundle by Lemma \ref{kno2}(iv) and we get case (ii). In case (2) we can assume that $m \ge 3$ for otherwise we are in case (1). We know that $\dim F_m(X,x)= 1$, hence $r=2$ and $\varphi(\P(\E))$ has dimension $\nu(\E)=m+1$. Moreover $\varphi(\P(\E))$ is covered by a family of dimension $2m$ of lines $\{P_x, x \in X\}$: the family has dimension $2m$ because $c_1(\E)^n>0$, hence $\Phi$ is birational onto its image. Therefore $\varphi(\P(\E))=\P^{m+1} = \P H^0(\E)=\P^{2d-1}$ where $d = \deg X$. But $d= \frac{(2m)!}{m!(m+1)!}$ giving the contradiction $m+1=2d-1$.

Assume now that we are in case (i).

If $b=0$ then $(X,\O_X(1)) \cong (\P^n,\O_{\P^n}(1))$ and Remark \ref{kno1} gives that $\E \cong \O_{\P^n}^{\oplus r}$, hence $\nu(\E)=r-1$.

Now suppose that $b \ge 1$. By \cite[Main Thm.]{sa2} we have that $\Pi_y \subseteq \P^{n-b}$ for a general $y$.

If $c_1(\E)^n=0$ we have by \cite[Cor.~3]{ls} that $\Pi_y$ is a general fiber $F$ of $\Phi$ and therefore $\Phi_{|\P^{n-b}}$ must be constant. Thus $F=\Pi_y=\P^{n-b}, \nu(\E)=b+r-1$ and $\det \E$ is trivial on the fibers of $p$. Hence there is a globally generated line bundle $M$ on $B$ such that $\det \E = p^*M$ and $\E$ is as in (i1) by \cite[Lemmas 5.1 and 4.1]{lo}. 

If $c_1(\E)^n>0$ we know by \cite[Cor.~2]{ls} that through a general point $x \in X$ we can find infinitely many linear spaces  
$\Pi_y$, hence $n+r-1-\nu(\E) \le n-b-1$, that is $\nu(\E) \ge b+r$ and then $b \le \frac{n}{2}-1$. Also, if $\nu(\E) = b+r$ we have, for every $y \in \varphi(\P(\E))$,
$$\dim \Pi_y \ge n+r-1-\nu(\E) = n-b-1 \ge b+1.$$ 
Since $\Pi_y$ is a linear space, we deduce that $p_{|\Pi_y} : \Pi_y \to B$ is constant. Thus we can apply Lemma \ref{nu=b+r}(ii) and find that $\E_{|f}\cong T_{\P^{n-b}}(-1)\oplus \O_{\P^{n-b}}^{\oplus (r-n+b)}$. This gives case (i2).
\end{proof}
\renewcommand{\proofname}{Proof}

\end{document}